\documentclass{article}
\title{A new proof of the Bondal-Orlov reconstruction \\ using Matsui spectra}
\author{Daigo Ito and Hiroki Matsui \thanks{The second author was partly supported by JSPS Grant-in-Aid for Early-Career Scientists 22K13894.}}
\date{}

\usepackage{stix} 
\usepackage[english]{babel} 
\usepackage{amsmath,amsfonts,amsthm,amscd}
\usepackage{ascmac}
\usepackage{appendix}
\usepackage{bm}
\usepackage{cases}
\usepackage{changepage}
\usepackage{enumitem}
\usepackage{etoolbox}
\usepackage{float}
\usepackage{geometry}
\usepackage{graphicx}
\usepackage[utf8]{inputenc}
\usepackage [autostyle, english = american]{csquotes}
\usepackage{scalerel}
\usepackage{ytableau}
\usepackage{tikz-cd} 
\usepackage{tocloft}
\usepackage{setspace}
\usepackage[color,all]{xy}
\usepackage[cal=euler]{mathalfa}
\usepackage{url}
\usepackage{CJKutf8}
\makeatletter
\newcommand{\address}[1]{\gdef\@address{#1}}
\newcommand{\email}[1]{\gdef\@email{\url{#1}}}
\newcommand{\website}[1]{\gdef\@website{\url{#1}}}
\newcommand{\@endstuff}{\par\vspace{\baselineskip}\noindent\small
\begin{tabular}{@{}l}\scshape{Daigo Ito} \\ \scshape\@address\\\textrm{E-mail address:} \@email \\\textrm{Website:} \@website 
\\
\\
\scshape{Hiroki Matsui} \\ \scshape{Department of Mathematical Sciences, Tokushima University, Tokushima 770-8506, Japan}\\\textrm{E-mail address:} \url{hmatsui@tokushima-u.ac.jp} \\\textrm{Website:} \url{https://mthiroki.github.io/}\end{tabular}}
\AtEndDocument{\@endstuff}
\makeatother
\address{Department of Mathematics, University of California, Berkeley, Evans Hall, CA 94720-3840, USA}
\email{daigoi@berkeley.edu}
\website{https://daigoi.github.io/}
\usepackage[alphabetic,backrefs]{amsrefs}
\usepackage{hyperref}
\hypersetup{colorlinks,linkcolor=blue,citecolor=red}

\DeclareMathOperator{\aut}{Aut}

\DeclareMathOperator{\auteq}{Auteq}

\DeclareMathOperator{\End}{End}

\DeclareMathOperator{\FM}{FM}

\let \ker \relax
\DeclareMathOperator{\ker}{Ker}

\DeclareMathOperator{\perf}{\mathsf{Perf}}

\DeclareMathOperator{\pic}{Pic}

\DeclareMathOperator{\spc}{Spc}
\DeclareMathOperator{\spec}{Spec}

\let \sp \relax
\DeclareMathOperator{\sp}{Sp}

\DeclareMathOperator{\Th}{Th}

\newcommand {\bb}{\mathbb}
\renewcommand {\cal}{\mathcal}
\newcommand{\ecal}{\mathscr}
\renewcommand {\epsilon}{\varepsilon}

\renewcommand {\l}{\left}

\renewcommand {\r}{\right}

\newcommand {\emp}{\emptyset}

\newcommand {\inj}{\hookrightarrow}
\newcommand {\inv}{^{-1}}
\newcommand {\iso}{\cong}

\newcommand {\tens}{\otimes}

\newsavebox{\pullbacks}
\sbox\pullbacks{%
\begin{tikzpicture}%
\draw (0,0) -- (1ex,0ex);%
\draw (1ex,0ex) -- (1ex,1ex);%
\end{tikzpicture}}

\newcommand{\fm}{\mathsf{FM}}

\newcommand {\id}{{\rm id}}

\newcommand{\ser}{\mathsf{Ser}}
 
\let \sf \relax 
\newcommand{\sf}{\mathsf}

\newcommand{\injar}{\ar@{^(->}} 
\newcommand{\prarrow}[2]{\ar@<0.5ex>[r]^-{#1} \ar@<-0.5ex>[r]_-{#2}}
\newcommand{\plarrow}[2]{\ar@<0.5ex>[l]^-{#1} \ar@<-0.5ex>[l]_-{#2}}
\newcommand{\pdarrow}[2]{\ar@<0.5ex>[d]^-{#1} \ar@<-0.5ex>[d]_-{#2}}
\newcommand{\puarrow}[2]{\ar@<0.5ex>[u]^-{#1} \ar@<-0.5ex>[u]_-{#2}}

\theoremstyle{plain}
\newtheorem{theorem}{Theorem}[section] 
\newtheorem{claim}{Claim}[theorem]
\newtheorem{conjecture}[theorem]{Conjecture}
\newtheorem{corollary}[theorem]{Corollary}
\newtheorem{lemma}[theorem]{Lemma} 
\newtheorem{prop}[theorem]{Proposition}

\newtheorem{penmdef}[claim]{Definition}

\theoremstyle{definition}
\newtheorem{construction}[theorem]{Construction}
\newtheorem{definition}[theorem]{Definition} 
 
\newtheorem{example}[theorem]{Example}     
\newtheorem{notation}[theorem]{Notation}

\newtheorem{penmlem}[claim]{Lemma}
\newtheorem{penmthm}[claim]{Theorem}
\newtheorem{penmcor}[claim]{Corollary}
\newtheorem{penmeg}[claim]{Example} 
\newtheorem{inclaims}[claim]{Claim}

\theoremstyle{remark}
\newtheorem{remark}[theorem]{Remark}

\newtheorem{penmrem}[claim]{Remark}

\makeatletter
\newtheoremstyle{indented}
  {1pt}
  {1pt}
  {\addtolength{\@totalleftmargin}{1.5em}
   \addtolength{\linewidth}{-1.5em}
   \parshape 1 1.5em \linewidth}
  {}
  {\bfseries}
  {.}
  {.5em}
  {}
\makeatother

\theoremstyle{indented}

\MakeOuterQuote{"}

\renewcommand\footnotemark{} 

\begin{document}
\maketitle

\begin{abstract}
In 2005, Balmer defined the ringed space $\operatorname{Spec}_\otimes \mathcal{T}$ for a given tensor triangulated category, while in 2023, the second author introduced the ringed space $\operatorname{Spec}_\vartriangle \mathcal{T}$ for a given triangulated category. In the algebro-geometric context, these spectra provided several reconstruction theorems using derived categories. In this paper, we prove that $\operatorname{Spec}_{\otimes_X^\mathbb{L}} \operatorname{Perf} X$ is an open ringed subspace of $\operatorname{Spec}_\vartriangle \operatorname{Perf} X$ for a quasi-projective variety $X$. As an application, we provide a new proof of the Bondal-Orlov and Ballard reconstruction theorems in terms of these spectra.

Recently, the first author introduced the Fourier-Mukai locus $\operatorname{Spec}^\mathsf{FM} \operatorname{Perf} X$ for a smooth projective variety $X$, which is constructed by gluing Fourier-Mukai partners of $X$ inside $\operatorname{Spec}_\vartriangle \operatorname{Perf} X$. As another application of our main theorem, we demonstrate that $\operatorname{Spec}^\mathsf{FM} \operatorname{Perf} X$ can be viewed as an open ringed subspace of $\operatorname{Spec}_\vartriangle \operatorname{Perf} X$. As a result, we show that all the Fourier-Mukai partners of an abelian variety $X$ can be reconstructed by topologically identifying the Fourier-Mukai locus within $\operatorname{Spec}_\vartriangle \operatorname{Perf} X$.

\end{abstract}

\tableofcontents
\section{Introduction}
We provide a new proof to the following version of the reconstruction theorem of Bondal-Orlov (\cite{bondal_orlov_2001}) shown by Ballard (\cite{ballard2011derived}). 
\begin{theorem}\label{thm: bondal-orlov-ballard}
    Let $X$ be a Gorenstein projective variety over an algebraically closed field with (anti-)ample canonical bundle. Then, the following assertions hold:
    \begin{enumerate}
        \item The variety $X$ can be reconstructed solely from the triangulated category structure of the derived category $\perf X$ of perfect complexes on $X$.
        \item If there exists a Gorenstein projective variety $Y$ with $\perf X \simeq \perf Y$, then $X \iso Y$. 
    \end{enumerate} 
\end{theorem}
The basic idea of the proof is that using the Serre functor, we can reconstruct the Balmer spectrum. To provide a more detailed sketch, we recall the following constructions. In the sequel, let us assume $X$ is a noetherian scheme unless otherwise specified. 
\begin{itemize}
    \item In \cite{Balmer_2002}, Balmer constructed a ringed space $$\spec_{\tens_X^\bb L} \perf X = (\spc_{\tens_X^\bb L} \perf X, \ecal O_{\perf X, \tens}),$$ called the \textbf{Balmer spectrum}, from the tensor triangulated category $(\perf X, \tens_X^\bb L)$, where we set $\tens_X^\bb L:= \tens_{\ecal O_X}^\bb L$. This construction provides the seminal reconstruction result:
    \[
    \spec_{\tens_X^\bb L} \perf X \iso X. 
    \]
    \item In \cites{Matsui_2021, matsui2023triangular}, one of the authors constructed a ringed space $$\spec_\vartriangle \perf X = (\spc_\vartriangle \perf X, \ecal O_{\perf X, \vartriangle}),$$ called the \textbf{Matsui spectrum}\footnote{In these papers, the author introduced the ringed space under the name `triangular spectrum', and it is called the Matsui spectrum in \cites{HO22,hirano2024FMlocusK3,ito2023gluing}.}, only using the triangulated category structure of $\perf X$. 
\end{itemize}
Note that for each spectrum, the underlying topological space consists of thick subcategories of $\perf X$ satisfying certain conditions. For comparisons of those two spectra, there are three key results from \cites{Matsui_2021, matsui2023triangular, HO22}, respectively:
\begin{itemize}
    \item The Balmer spectrum is a subspace of the Matsui spectrum, topologically:
    \[
    \spc_{\tens_X^\bb L} \perf X \subset \spc_\vartriangle \perf X;
    \]
    \item If $\spc_{\tens_X^\bb L} \perf X \subset \spc_\vartriangle \perf X$ is an open subspace, then there is an open immersion of ringed spaces
    \[
    (\spec_{\tens_X^\bb L} \perf X)_\sf{red} \inj \spec_\vartriangle \perf X;
    \]
    \item Suppose $X$ is a Gorenstein projective variety with (anti-)ample canonical bundle. Then, we have
    \[
    \spc_{\tens_X^\bb L} \perf X = \spc^\ser \perf X \subset \spc_\vartriangle \perf X
    \]
    where we let $\bb S$ denote the Serre functor of $\perf X$ and define the \textbf{Serre invariant locus} to be 
    \[
    \spc^\ser \perf X := \{\cal P \in \spc_\vartriangle \perf X \mid \bb S(\cal P) = \cal P\} \subset \spc_\vartriangle \perf X.
    \]
    Here, since each point in $\spc_\vartriangle \perf X$ is a certain thick subcategory of $\perf X$, we see that the notation indeed makes sense. Note in particular that the underlying topological space of the Balmer spectrum is determined solely by the triangulated category structure of $\perf X$ in this case. 
\end{itemize}

Therefore, the following result is enough to complete the proof of Theorem \ref{thm: bondal-orlov-ballard}. 
\begin{theorem}[Theorem \ref{theorem: Balmer is open in Matsui}]\label{thm: intro open}
    If $X$ is a quasi-projective scheme over an algebraically closed field, then the inclusion
    \[
    \spc_{\tens_X^\bb L} \perf X \subset \spc_\vartriangle \perf X
    \]
    is open. In particular, we have an isomorphism of ringed spaces:
    \[
    (\spc_{\tens_X^\bb L} \perf X, \ecal O_{\perf X, \vartriangle}|_{\spc_{\tens_X^\bb L} \perf X}) \iso \spec_{\tens_X^\bb L} \perf X \iso X.
    \]
\end{theorem}
Indeed, combining the theorem with aforementioned results, we see that if $X$ is a Gorenstein projective variety with (anti-)ample canonical bundle, then we can reconstruct $X$ by restricting the structure sheaf of the Matsui spectrum to the Serre invariant locus, i.e., 
    \[
    X \iso (\spc^\ser \perf X, \ecal O_{\perf X, \vartriangle}|_{\spc^\ser \perf X})
    \]
where the right-hand side only depends on the triangulated category structure of $\perf X$.  

Now, Theorem \ref{thm: intro open} has more consequences than just finishing up the new proof of the reconstruction theorem. In the rest of this paper, we will consider its implications in terms of the Fourier-Mukai locus introduced in \cite{ito2023gluing}. In particular, we observe that the structure sheaf on the Fourier-Mukai locus constructed in \cite{ito2023gluing} can be simply realized as the restriction of the structure sheaf of the Matsui spectrum (Theorem \ref{prop: structure sheaf of FM locus}). In other words, there is an open immersion
\[
\spec^\fm \perf X \inj \spec_\vartriangle \perf X
\]
of ringed spaces, where $\spec^\fm \perf X$ can be viewed as a scheme constructed by gluing copies of Fourier-Mukai partners of $X$ realized as Balmer spectra inside the Matsui spectrum. Let us note that the Fourier-Mukai locus contains various (birational) geometric information about the Fourier-Mukai partners as observed in \cite{ito2023gluing} (cf. Example \ref{example: FM locus}) and hence the categorical construction of the structure sheaf gives more paths to applications to geometry. 

Finally, we give an affirmative answer to the following conjecture \cite{ito2023gluing}*{\href{https://arxiv.org/pdf/2309.08147.pdf}{Conjecture 5.9}} on the Fourier-Mukai locus of an abelian variety. 
\begin{theorem}[Theorem \ref{main theorem: abelian}] \label{conj: FM locus paper}
Let $X$ be an abelian variety. Then, the Fourier-Mukai locus $\spec^\fm \perf X$ is the disjoint union of copies of Fourier-Mukai partners of $X$. In particular, connected components of $\spec^\fm \perf X$ are precisely the Fourier-Mukai partners of $X$. 
\end{theorem}
In particular, this result together with Theorem \ref{thm: intro open} tells us that if we can identify the underlying topological space of the Fourier-Mukai locus inside the Matsui spectrum purely categorically, we can reconstruct all the Fourier-Mukai partners of abelian varieties as connected components of the Fourier-Mukai locus. 
\section{Preliminaries}
\begin{notation}
    Let $k$ be a field. We assume $k$ is algebraically closed unless otherwise specified. Throughout this paper, a triangulated category is assumed to be $k$-linear and \textbf{essentially small} (i.e., having a set of isomorphism classes of objects) and functors/structures are assumed to be $k$-linear. A variety is an integral scheme of finite type over $k$. {For a variety, points refer to closed points.} Moreover, any ring and scheme are assumed to be over $k$ and for a scheme $X$, let $\perf X$ denote the derived category of perfect complexes on $X$ and let
    \[
    \tens_X^\bb L : = \tens_{\ecal O_X}^\bb L
    \] 
    denote the usual derived tensor product on $\perf X$. 
\end{notation}
\begin{definition} Let $\cal T$ be a triangulated category.
\begin{enumerate}
    \item Let $\Th \cal T$ denote the poset of thick subcategories of $\cal T$ by inclusions. We define the \textbf{Balmer topology} on $\Th \cal T$ by setting open sets to be
    \[
    U(\cal E) := \{\cal I \in \Th \cal T \mid \cal I \cap \cal E \neq \emp\}
    \]
    for each collection $\cal E$ of objects in $\cal T$. 
    \item We say a thick subcategory $\cal P$ is a \textbf{prime thick subcategory} if the subposet 
    \[
    \{\cal I \in \Th \cal T \mid \cal I \supsetneq \cal P\} \subset \Th \cal T
    \]
    has the smallest element. Define the \textbf{Matsui spectrum}
    \[
    \spc_\vartriangle \cal T \subset \Th \cal T
    \]
    to be the subspace consisting of prime thick subcategories. We can equip the Matsui spectrum with a ringed space structure and let $\spec_\vartriangle \cal T = (\spc_\vartriangle \cal T,\ecal O_{\cal T,\vartriangle})$ denote the ringed space. 
    {Here, the structure sheaf $\ecal O_{\cal T,\vartriangle}$ is defined as the sheafification of the presheaf
    \[
    \spc_\vartriangle \cal T \supseteq U \mapsto Z\left(\cal T/\bigcap_{\cal P \in U} \cal P\right),
    \]
    where
    $$
    Z(\cal S):=\{\text{natural transformations $\eta:\id_{\cal S} \to \id_{\cal S}$ with $\eta[1] = [1]\eta$}\}
    $$ 
    denotes the center of $\cal S$ for a triangulated category $\cal S$. See \cite{matsui2023triangular} for the detailed construction.}
    
    \item We say a symmetric monoidal category $(\cal T,\tens{,\mathbf{1}})$ is a \textbf{tensor triangulated category} (\textbf{tt-category} in short) if the bifunctor $\tens:\cal T \times \cal T \to \cal T$ is triangulated in each variable, which is called a \textbf{tt-structure} on $\cal T$. A thick subcategory $\cal I \subset \cal T$ is said to be a \textbf{$\tens$-ideal} if, for any $F \in \cal T$ and $G \in \cal I$, we have $F\tens G \in \cal I$. A $\tens$-ideal $\cal P$ is said to be a \textbf{prime $\tens$-ideal} if $F \tens G \in \cal P$ implies $F \in \cal P$ or $G \in \cal P$. Define the \textbf{Balmer spectrum} 
    \[
    \spc_\tens \cal T \subset \Th \cal T
    \]
    to be the subspace consisting of prime $\tens$-ideals of $(\cal T, \tens)$. 
    We can equip the Balmer spectrum with a ringed space structure and let $\spec_\tens \cal T = (\spc_\tens \cal T,\ecal O_{\cal T, \tens})$ denote the ringed space. 
    {Here, the structure sheaf $\ecal O_{\cal T, \otimes}$ is defined as the sheafification of the presheaf
    \[
    \spc_\tens \cal T \supseteq U \mapsto \End_{\cal T/ \bigcap_{\cal P \in U}\cal P}(\mathbf{1}).
    \]
    See \cite{Balmer_2005} for the detailed construction.}
\end{enumerate}
\end{definition}
In the algebro-geometric setting, we have the following results:
\begin{theorem}[\cite{Balmer_2005}*{\href{https://arxiv.org/pdf/math/0409360.pdf}{Theorem 6.3}}, \cite{matsui2023triangular}*{\href{https://arxiv.org/pdf/2301.03168.pdf}{Theorem 2.12, Corollary 4.7}}]\label{thm: matsui23}
    Let $X$ be a noetherian scheme (over $\bb Z$). Then, we have the following assertions:
    \begin{enumerate}
        \item There is a canonical isomorphism $X \iso \spec_{\tens_X^\bb L} \perf X$ of ringed spaces whose underlying map is given by sending a (not necessarily closed) point $x \in X$ to a thick subcategory
        \[
        \cal S_X(x):= \{\ecal F \in \perf X \mid \ecal F_x \iso 0 \text{ in $\perf \ecal O_{X,x}$}\} \subset \perf X;
        \]
        \item {A $\tens_X^\bb L$-ideal of $\perf X$ is a prime thick subcategory if and only if it is a prime $\tens_X^\bb L$-ideal.}
        \item There is a morphism 
        \[
        i: \spec_{\tens_X^\bb L}\perf X \to \spec_\vartriangle \perf X
        \]
        of ringed spaces whose underlying continuous map is the inclusion;
        \item Suppose $\spc_{\tens_X^\bb L}\perf X$ is open in $\spc_\vartriangle \perf X$. Then, the morphism 
        \[
        i: (\spec_{\tens_X^\bb L}\perf X)_\sf{red} \to \spec_\vartriangle \perf X
        \]
        is an open immersion of ringed spaces. 
    \end{enumerate}
\end{theorem}
In \cite{matsui2023triangular}, it is shown that a quasi-affine scheme satisfies the supposition in part (iv). In the next section, we show that the supposition holds more generally if $X$ is a quasi-projective scheme over $k$. 
\begin{remark}\label{remark:k-linear}
    Let $X$ be a noetherian scheme over $k$. We observe that Theorem \ref{thm: matsui23} extends to this relative setting in the following sense.
    \begin{enumerate}
        \item By definition, the structure sheaf $\ecal O_{\perf X, \tens_X^\bb L}$ of the Balmer spectrum is naturally a sheaf of $k$-algebras and therefore the scheme $\spec_{\tens_X^\bb L}\perf X$ has a canonical $k$-scheme structure. 
        Now, since for a $k$-algebra $A$, the canonical ring isomorphism $$\End_{\perf A}(A) \iso \Gamma(\spec A, \ecal O_{\spec A})\iso A$$ respects $k$-algebra structures, we see that the canonical isomorphism in Theorem \ref{thm: matsui23} (i)
        \[
        \spec_{\tens_X^\bb L}\perf X \iso X
        \]
        is an isomorphism of $k$-schemes. 
        \item Similarly, {the center of a $k$-linear triangulated category has the canonical structure of a $k$-algebra so that} $\spec_\vartriangle \perf X$ is naturally a ringed space over $\spec k$. Now, for a $k$-algebra $A$, the canonical evaluation ring homomorphism
        \[
        Z(\perf A) \to \End_{\perf A}(A), \quad \eta \mapsto \eta_A
        \]
        respects $k$-algebra structures. 
        Since the open immersion $$i:(\spec_{\tens_X^\bb L}\perf X)_\sf{red} \inj \spec_\vartriangle \perf X$$ in Theorem \ref{thm: matsui23} (iv) is essentially coming from the evaluation maps above, we see that $i$ is an open immersion of ringed spaces over $\spec k$. 
    \end{enumerate}
\end{remark}
\section{A new proof of Bondal-Orlov reconstruction}
First of all, let us see the following observation essentially made in \cite{HO22}*{\href{https://arxiv.org/pdf/2112.13486}{Proposition 5.3}}.
\begin{lemma}\label{lemma: hirano-ouchi}
    Let $X$ be a quasi-projective scheme of dimension $n$ over $k$ and take a line bundle $\ecal L$ on $X$. Then, the following hold:
    \begin{enumerate}
        \item Take a thick subcategory $\cal I \subset  {\perf X}$ and an object $\ecal F \in \cal I$. If {$\ecal L$ is very ample and} there exists $d \in \bb Z$ such that 
        \[
        \ecal F \tens_X^\bb L \ecal L^{\tens d},\ \ecal F \tens_X^\bb L \ecal L^{\tens (d+1)},\ \dots,\ \ecal F \tens_X^\bb L \ecal L^{\tens (d+n)} \in \cal I, 
        \]
        then for any $\ecal G \in {\perf X}$, we have
        \[
        \ecal F \tens_X^\bb L \ecal G \in \cal I.
        \]
        \item {Assume that $\ecal L$ or $\ecal L^{\tens -1}$ is ample.} A prime thick subcategory $\cal P \in \spc_\vartriangle \perf X$ is a prime $\tens$-ideal if and only if 
        \[
        \cal P \tens_X^\bb L \ecal L = \cal P.
        \]
    \end{enumerate}
\end{lemma}
\begin{proof} For part (i), take a thick subcategory $\cal I \subset \perf X$ and an object $\ecal F \in \cal I$ and suppose that there exists $d \in \bb Z$ such that 
        \[
        \ecal F \tens_X^\bb L \ecal L^{\tens d},\ \ecal F \tens_X^\bb L \ecal L^{\tens (d+1)},\ \dots,\ \ecal F \tens_X^\bb L \ecal L^{\tens (d+n)} \in \cal I.
        \]
         {Then the subcategory $\cal X := \{\ecal G \in \perf X \mid \ecal F \tens_X^\bb L \ecal G \in \cal I\}$ is a thick subcategory of $\perf X$ containing 
        \[
        \ecal L^{\tens d},\ \ecal L^{\tens(d+1)},\ \dots,\ \ecal L^{\tens{(d+n)}}.
        \]
        By \cite{Orlov_dimension}*{\href{https://arxiv.org/pdf/0804.1163.pdf}{Theorem 4}}, a thick subcategory containing $\ecal L^{\tens d},\ \ecal L^{\tens(d+1)},\ \dots,\ \ecal L^{\tens{(d+n)}}$ needs to be $\perf X$. 
        Therefore, we see that $\cal X = \perf X$; that is 
        \[
        \ecal F \tens_X^\bb L \ecal G \in \cal I
        \]
        for any $\ecal G \in {\perf X}$.}
        
       For part (ii), note that for any $\tens$-ideal $\cal P$ in $(\perf X, \tens_{X}^\bb L)$, we clearly have
        \[
        \cal P \tens_X^\bb L \ecal L \subset \cal P
        \]
        and the inclusion is the equality since we also have $\cal P \tens_X^\bb L \ecal L^{\tens -1} \subset \cal P$. 
        If $\cal P \tens_X^\bb L \ecal L = \cal P$ holds, then $\cal P \tens_X^\bb L \ecal L^{\otimes d} = \cal P$ for any $d \in \bb Z$. Thus the converse follows from part (i) and Theorem \ref{thm: matsui23} (ii).
\end{proof}
Now, we show our main theorem on the topology of the Matsui spectrum.

\begin{theorem}\label{theorem: Balmer is open in Matsui}
    Let $X$ be a quasi-projective scheme of dimension $n$ over $k$. Then the inclusion
    \[
    \spc_{\tens_X^\bb L} \perf X \subset \spc_\vartriangle \perf X
    \]
    is open. In particular, we have an isomorphism of ringed spaces:
    \[
    (\spc_{\tens_X^\bb L} \perf X, \ecal O_{\perf X, \vartriangle}|_{\spc_{\tens_X^\bb L} \perf X}) \iso (\spec_{\tens_X^\bb L} \perf X)_{\sf{red}} \iso X_{\sf{red}}.
    \] 
\end{theorem}
\begin{proof}
    Take a very ample line bundle $\ecal L$ on $X$ and fix a corresponding immersion $X \inj \bb P^N_k$
    for some $N$. In this proof, we say an effective Cartier divisor $H \subset X$ is a \textbf{hyperplane section} if there exists a hyperplane $\tilde H \subset \bb P^N_k$ {with $X \not\subseteq \tilde H$} such that $H = X \cap \tilde H$. For each hyperplane section $H \subset X$, define 
    \[
    \ecal F_H:= \ecal O_H \oplus (\ecal O_H \tens_X^\bb L \ecal O_X(H)) \oplus \cdots \oplus (\ecal O_H\tens_X^\bb L\ecal O_X(nH))
    \]
    and set
    \[
    \cal E_X := \{\ecal F_H \mid \text{$H\subset X$ is a hyperplane section}\}.
    \]
    Now, to see $\spc_{\tens_X^\bb L} \perf X$ is open in $\spc_\vartriangle \perf X$, we are going to show 
    \[
    \spc_{\tens_X^\bb L} \perf X = U(\cal E_X) \overset{\sf{def}}{=} \{\cal I \in \spc_\vartriangle \perf X \mid  \cal I \cap \cal E_X \neq \emp \}.
    \]
    The containment $\spc_{\tens_X^\bb L} \perf X \subset U(\cal E_X)$ follows since for any (not necessarily closed) point $x \in X$, there exists a hyperplane section $x \not \in H \subset X$ (as $k$ is assumed to be algebraically closed and hence infinite) and therefore the prime $\tens_X^\bb L$-ideal $\cal S_X(x)$ contains $\ecal F_H$. Conversely, take a prime thick subcategory $\cal P \in U(\cal E_X)$ and take $\ecal F_H \in \cal P$ for some hyperplane section $H \subset X$. Then, in particular, we have 
    \[
    \ecal O_H, \ecal O_H \tens_X^\bb L \ecal O_X(H), \dots, \ecal O_H\tens_X^\bb L\ecal O_X(nH) \in \cal P.
    \]
    As the line bundle $\ecal O_X(H) \iso \ecal L$ is very ample, we see from Lemma \ref{lemma: hirano-ouchi} (i) that
    \[
    \ecal O_H \tens_X^\bb L \ecal F \in \cal P
    \]
    {holds for any $\ecal F \in \cal P$} and hence 
    \[
    \ecal O_X(-H) \otimes_{X}^\bb L \ecal F \in \cal P
    \]
    by considering a distinguished triangle 
    \[
    \ecal O_{X}(-H)\tens_X^\bb L \ecal F \to \ecal O_{X} \tens_X^\bb L \ecal F \to \ecal O_{H} \tens_X^\bb L \ecal F \to \ecal O_{X}(-H)\tens_X^\bb L \ecal F[1]. 
    \]
    By iterating this process, we see that for any $\ecal F \in \cal P$, one has
    \[
    \ecal O_X (-H) \tens_X^\bb L \ecal F,  \ecal O_X (-2H) \tens_X^\bb L \ecal F, \dots,  \ecal O_X (-(n+1)H) \tens_X^\bb L \ecal F \in \cal P
    \]
    and therefore again by Lemma \ref{lemma: hirano-ouchi} (i), we see that for any $\ecal G \in \perf X$, $\ecal G \tens_X^\bb L \ecal F \in \cal P$. Thus, $\cal P$ is a $\tens$-ideal and by Theorem \ref{thm: matsui23} (ii), $\cal P$ is a prime $\tens$-ideal, i.e., $\cal P \in \spc_{\tens_X^\bb L}\perf X$. The later claim follows by Theorem \ref{thm: matsui23} (iv). 
\end{proof}
By Theorem \ref{theorem: Balmer is open in Matsui}, to reconstruct a reduced quasi-projective scheme over $k$, it is sufficient to identify the {underlying topological space of} corresponding Balmer spectrum in the Matsui spectrum, which is indeed done in \cite{HO22} in the case of Gorenstein projective varieties with (anti-)ample canonical bundle. Thus, we can give a new conceptually simple proof of the following version of the Bondal-Orlov reconstruction shown by Ballard (\cite{ballard2011derived}*{\href{https://arxiv.org/pdf/0801.2599.pdf}{Theorem 6.1}}). 
\begin{theorem}[Bondal-Orlov, Ballard]
    Let $X$ be a Gorenstein projective variety over $k$ with (anti-)ample canonical bundle. Then, the following assertions hold:
    \begin{enumerate}
        \item The scheme $X$ can be reconstructed solely from the triangulated category structure of $\perf X$.
        \item If there exists a Gorenstein projective variety $Y$ with $\perf X \simeq \perf Y$, then $X \iso Y$. 
    \end{enumerate} 
\end{theorem}
\begin{proof}
    For part (i), note from \cite{HO22}*{\href{https://arxiv.org/pdf/2112.13486.pdf}{Corollary 5.4}} that we have
        \[
        \spc^\ser \perf X = \spc_{\tens_X^\bb L} \perf X \subset \spc_\vartriangle \perf X
        \]
        where $\spc^\ser \perf X$ denote the Serre invariant locus, i.e., the subspace of prime thick subcategories $\cal P$ satisfying $\bb S(\cal P) = \cal P$ for the Serre functor $\bb S$. (The equality is indeed a direct consequence of Lemma \ref{lemma: hirano-ouchi} (ii).) Now, by Theorem \ref{thm: matsui23} (i), (iv) and Theorem \ref{theorem: Balmer is open in Matsui}, we have
        \[
        X \iso {\spec_{\tens_X^\bb L} \perf X \iso } (\spc^\ser \perf X, \ecal O_{\perf X, \vartriangle}|_{\spc^\ser \perf X}),
        \]
        {where the right-most ringed space is determined by the triangulated category structure of $\perf X$.}
        
        For part (ii), take a Gorenstein projective variety $Y$ with $\Phi:\perf X \simeq \perf Y$. Then, we have {an open embedding}
        \[
        \Phi\inv(\spc_{\tens_Y^\bb L}\perf Y) \subset \spc^\ser \perf X
        \]
        by \cite{ito2023gluing}*{\href{https://arxiv.org/pdf/2309.08147.pdf}{Corollary 6.3}} {and Theorem \ref{theorem: Balmer is open in Matsui}}. 
        Now, by Theorem \ref{thm: matsui23} the following composition $f$ of canonical morphisms
        \begin{align*}
            Y \iso \spec_{\tens_Y^\bb L} \perf Y & \iso (\spc_{\tens_Y^\bb L}\perf Y, \ecal O_{\perf Y,\vartriangle}|_{\spc_{\tens_Y^\bb L}\perf Y}) \\
            &\iso (\Phi\inv(\spc_{\tens_Y^\bb L}\perf Y), \ecal O_{\perf X,\vartriangle}|_{\Phi\inv(\spc_{\tens_Y^\bb L}\perf Y)}) \\
            &\inj (\spc^\ser \perf X, \ecal O_{\perf X, \vartriangle}|_{\spc^\ser \perf X}) \iso \spec_{\tens_X^\bb L} \perf X  \iso X
        \end{align*}
        is an open immersion of ringed spaces. Moreover, by Remark \ref{remark:k-linear} and by the fact that $k$-linear triangulated equivalence induces a $k$-isomorphism of the Matsui spectra, we see that $f$ is an open immersion of $k$-schemes. Now, since $f$ is moreover proper, we see that $f$ is a closed and open immersion, and therefore $f: Y\to X$ is an isomorphism as $X$ is irreducible and hence connected. 
\end{proof}
By using similar ideas, we can get the following results as well.

\begin{corollary}
    Let $X$ be a reduced quasi-affine scheme over $k$. Then, the following assertions hold:
    \begin{enumerate}
        \item The scheme $X$ can be reconstructed solely from the triangulated category structure of $\perf X$.
        \item If there exists a noetherian reduced scheme $Y$ over $k$ with $\perf X \simeq \perf Y$, then $X \iso Y$. 
    \end{enumerate} 
\end{corollary}
\begin{proof} For part (i), since $X$ is quasi-affine (hence $\ecal O_X$ is ample), we get an isomorphism $$\spec_\vartriangle \perf X \cong \spec_{\tens_X^\bb L}\perf X \cong X$$ by Theorem \ref{thm: matsui23}.

For part (ii), take a noetherian reduced scheme $Y$ over $k$ with $\Phi:\perf X \simeq \perf Y$. Then, we have an open immersion
        \[
        Y \hookrightarrow \spec_{\tens_X^\bb L}\perf Y \cong \spec_{\tens_X^\bb L}\perf X \cong X
        \]
        by Theorem \ref{theorem: Balmer is open in Matsui} and part (i). 
        In particular, $Y$ is also quasi-affine.
        Thus, part (i) shows that $X \cong \spec_\vartriangle \perf X \cong \spec_\vartriangle \perf Y \cong Y$.
\end{proof}

\begin{remark}
Favero \cite{FAVERO20121955}*{Corollary 3.11} proved the same result under the assumptions that $X$ is a quasi-affine variety and $Y$ is a divisorial variety. 
\end{remark}

{
\begin{remark}

The arguments of the above two corollaries also prove the following more general statement:
\begin{quote}
Let $X$ and $Y$ be quasi-projective schemes over $k$.
Let $\ecal L$ and $\ecal M$ be line bundles on $X$ and $Y$, respectively.
Assume that the following two conditions:
\begin{enumerate}
\item[\rm(a)]
$\ecal L$ is (anti-)ample.
\item[\rm(b)] 
There is a triangulated equivalence $\Phi: \perf X \simeq \perf Y$ such that $\Phi(\ecal F \tens_{X}^\bb L \ecal L) \cong \Phi(\ecal F) \tens_{Y}^\bb L \ecal M$ for every $\ecal F \in \perf X$.
\end{enumerate}
Then there is an open immersion $X \hookrightarrow Y$ of $k$-schemes.
\end{quote}
Further generalizations are discussed in a work in preparation (\cite{ito2024polarization}).
\end{remark}}

\section{Categorical construction of scheme structure on Fourier-Mukai locus}
In \cite{ito2023gluing}, one of the authors studied the following locus in the Matsui spectrum.
\begin{definition}
    Let $\cal T$ be a triangulated category. 
    \begin{enumerate}
        \item We say a smooth projective variety $X$ is a \textbf{Fourier-Mukai partner} of $\cal T$ if there exists a triangulated equivalence $\cal T \simeq \perf X$. Let $\FM \cal T$ denote the set of isomorphism classes of Fourier-Mukai partners of $\cal T$. 
        \item We say a tt-structure $\tens$ on $\cal T$ is \textbf{geometric in $X \in \FM \cal T$} if there exists an equivalence 
        \[
        (\cal T, \tens) \simeq (\perf X,\tens_{X}^\bb L)
        \]
        of tt-categories. We say a tt-structure $\tens$ on $\cal T$ is \textbf{geometric} if there exists $X \in \FM \cal T$ such that $\tens$ is geometric in $X$. 
        \item Define the \textbf{Fourier-Mukai locus} of $\cal T$ to be the subspace
        \[
        \spc^\fm \cal T := \bigcup_{\text{geom. tt-str. $\tens$ on $\cal T$}} \spc_\tens \cal T \subset \spc_\vartriangle \cal T.
        \]
    \end{enumerate}
\end{definition}
Now, we have the following consequences of Theorem \ref{theorem: Balmer is open in Matsui}.
\begin{prop}
    Let $\cal T$ be a triangulated category with $\FM \cal T \neq \emp$. Then, the following hold: 
    \begin{enumerate}
        \item For a geometric tt-structure $\tens$ on $\cal T$, the inclusion $\spc_\tens \cal T\subset \spc_\vartriangle \cal T$ is open. 
        \item The Fourier-Mukai locus of $\cal T$ is open in the Matsui spectrum of $\cal T$. 
        \item In \cite{ito2023gluing}, the topology on the Fourier-Mukai locus is defined to be the one generated by open subsets of the Balmer spectrum for each geometric tt-structure. This topology on $\spc^\fm \cal T$ agrees with the subspace topology on $\spc^\fm \cal T$ in $\spc_\vartriangle \cal T$. 
    \end{enumerate}   
\end{prop}
\begin{proof}
    {For a fixed $X \in \FM \cal T$ and a tt-equivalence $\Phi: \cal T \xrightarrow{\sim} \perf X$, we have a commutative diagram
    $$
    \xymatrix{
    \spc_\tens \cal T \ar[r]^\subseteq \ar[d]_\Phi^\cong & \spc^\fm \cal T \ar[r]^\subseteq \ar[d]_\Phi^\cong & \spc_\vartriangle \cal T \ar[d]_\Phi^\cong\\
    \spc_\tens \cal \perf X \ar[r]^\subseteq & \spc^\fm \cal \perf X \ar[r]^\subseteq & \spc_\vartriangle \cal \perf X
    }
    $$
    where the vertical maps induced from $\Phi$ are homeomorphisms.}
    By Theorem \ref{theorem: Balmer is open in Matsui}, we have that the inclusion $\spc_{\tens_X^\bb L} \perf X \subset \spc_\vartriangle \perf X$ is an open embedding. {Therefore, we have part (i).} Part (ii) and (iii) immediately follow from part (i). 
\end{proof}
Note by \cite{ito2023gluing}*{\href{https://arxiv.org/pdf/2309.08147.pdf}{Theorem 4.7}}, we can glue the Balmer spectra corresponding to geometric tt-structures to equip $\spc^\fm \cal T$ with a scheme structure, where the corresponding scheme is denoted by $\spec^\fm \cal T$. Recall by construction the scheme structure on the Fourier-Mukai locus satisfies the following properties. 
\begin{theorem}[\cite{ito2023gluing}*{\href{https://arxiv.org/pdf/2309.08147.pdf}{Theorem 4.7}}]\label{thm: gluing fm locus}
    Let $\cal T$ be a triangulated category with {$\FM \cal T \neq \emp$}. Then, the scheme $\spec^\fm \cal T$ is a smooth scheme locally of finite type and for any geometric tt-structure $\tens$, and we have a canonical open immersion
    \[
    \spec_\tens \cal T \inj \spec^\fm \cal T
    \]
    of schemes whose underlying continuous map is the inclusion. 
\end{theorem}
Now, we show that we can obtain the same scheme structure on the Fourier-Mukai locus by simply restricting the structure sheaf on the Matsui spectrum to the Fourier-Mukai locus. To see this, let us recall the following classical result (e.g. {\cite{EGAIV}*{Proposition 10.9.6} and} \cite{Har77}*{Proposition I.3.5}):
\begin{lemma}\label{lem: classical ag}
    Let $X$ and $Y$ be reduced schemes locally of finite type over an algebraically closed field $k$ and let $f,g:X\to Y$ be morphisms of schemes over $k$. If $f$ and $g$ agree on the set of closed points, then they agree as a morphism of schemes over $k$. 
\end{lemma}
Now, we are ready to show the following. 
\begin{theorem}\label{prop: structure sheaf of FM locus}
    Let $\cal T$ be a triangulated category with $\FM \cal T \neq \emp$. Then, there is an isomorphism
    \[
    \spec^\fm \cal T \overset{\sim}{\to} (\spc^\fm \cal T,\ecal O_{\cal T,\vartriangle}|_{\spc^\fm \cal T})
    \]
     of ringed spaces whose underlying continuous map is the identity. 
\end{theorem}
\begin{proof}
First, note that by {Theorem \ref{theorem: Balmer is open in Matsui}}, $(\spc^\fm \cal T,\ecal O_{\cal T,\vartriangle}|_{\spc^\fm \cal T})$ is a smooth scheme locally of finite type. Take an open covering $\{\spc_{\tens_\alpha} \cal T\}_{\alpha \in A}$ of $\spc^\fm \cal T$, where each $\tens_\alpha$ is a geometric tt-structure on $\cal T$. Then, for each $\alpha\in A$, there exists a canonical open immersion
\[
\iota_\alpha: \spec_{\tens_\alpha} \cal T \inj \spec^\fm \cal T 
\]
whose underlying continuous map is the inclusion by Theorem \ref{thm: gluing fm locus} and therefore there is an open immersion
\[
i_\alpha: \iota_{\alpha}(\spec_{\tens_\alpha}\cal T) \inj (\spc^\fm \cal T,\ecal O_{\cal T,\vartriangle}|_{\spc^\fm \cal T})
\]
whose underlying continuous map is the inclusion by {Theorem \ref{theorem: Balmer is open in Matsui}}. Now, by Lemma \ref{lem: classical ag}, we see that $\{i_\alpha\}_{\alpha \in A}$ glues to a morphism
\[
\phi:\spec^\fm \cal T \to (\spc^\fm \cal T,\ecal O_{\cal T,\vartriangle}|_{\spc^\fm \cal T})
\]
whose underlying continuous map is the identity. Since $\phi$ is a homeomorphism and locally an isomorphism, it is an isomorphism. 
\end{proof}
\begin{remark}\label{rem: topological reconstruction is enough}
    Note that Theorem \ref{prop: structure sheaf of FM locus} makes the construction of the structure sheaf on the Fourier-Mukai locus purely triangulated categorical. Therefore, if we can determine the underlying topological space of the Fourier-Mukai locus categorically, then we can reconstruct information coming from the gluings of structure sheaves of Fourier-Mukai partners performed in \cite{ito2023gluing}. In particular, we have more hope to have backward applications of the Fourier-Mukai locus to birational geometry.
\end{remark}

Keeping this remark in mind, let us recall some geometric results on the Fourier-Mukai locus from \cite{ito2023gluing}. First, we recall basic notations and terminologies. 
\begin{definition}
    Let $\cal T$ be a triangulated category with $X \in \FM \cal T$. 
    \begin{enumerate}
        \item Let $\spec_{\tens,X} \cal T \subset \spec^\fm\cal T$ denote the open subscheme whose underlying topological space $\spc_{\tens,X} \cal T$ is the union of the Balmer spectra corresponding to geometric tt-structures in $X$. In other words, set
        \[
        \spec_{\tens,X} \cal T := \bigcup_{\text{geom. tt-str. $\tens$ in $X$}} \spec_{\tens} \cal T \subset \spec^\fm \cal T.  
        \]
        Note that for a fixed geometric tt-structure $\tens_0$ in $X$, we can write
        \[
        \spc_{\tens,X} \cal T = \bigcup_{\tau \in \auteq \cal T} \tau(\spc_{\tens_0}\cal T),
        \]
        where $\auteq \cal T$ denote the group of natural isomorphism classes of triangulated autoequivalences of $\cal T$ and {$\tau(\spc_{\tens_0}\cal T)$ is the image of $\spc_{\tens_0}\cal T$ under the following action of $\auteq \cal T$ on $\spec_\vartriangle \cal T$:} 
        \[
        \auteq \cal T \times \spec_\vartriangle \cal T \ni (\tau, \cal P) \mapsto \tau(\cal P) \in \spec_\vartriangle \cal T. 
        \]
        \item We say $X$ is \textbf{tt-separated} (resp. \textbf{tt-irreducible}) if $\spec_{\tens,X} \cal T$ is separated (resp. irreducible). 
    \end{enumerate}
\end{definition}
It is natural to ask how those copies of Fourier-Mukai partners interact with each other in the Matsui spectrum. Indeed, the following results and examples in \cite{ito2023gluing} show that the topology of the Fourier-Mukai locus is closely related to types of possible equivalences between Fourier-Mukai partners, which are then related to (birational) geometric properties of varieties. 
\begin{definition}
    Let $X$ and $Y$ be smooth projective varieties. We say a triangulated equivalence
    \[
    \Phi:\perf X \to \perf Y
    \]
    is \textbf{birational} if there exists a closed point $x \in X$ such that $\Phi(k(x))$ is isomorphic to $k(y)$ for some closed point $y \in Y$, which implies that
    \[
    \Phi(\spc_{\tens_X^\bb L} \perf X) \cap \spc_{\tens_{\ecal O_Y}^\bb L} \perf Y \neq \emp
    \]
    and in particular that $X$ and $Y$ are birationally equivalent (and indeed $K$-equivalent) (cf. \cite{ito2023gluing}*{\href{https://arxiv.org/pdf/2309.08147.pdf}{Lemma 4.11}}). 
\end{definition}
We can characterize the topology of the Fourier-Mukai locus by using birational autoequivalences as follows. 
\begin{lemma}[\cite{ito2023gluing}*{\href{https://arxiv.org/pdf/2309.08147.pdf}{Corollary 4.21}}]\label{lem: tt-sep}
Let $X$ be a smooth projective variety. The following are equivalent:
\begin{enumerate}
    \item $X$ is tt-separated;
    \item $\spec_{\tens,X} \perf X$ is a disjoint union of copies of $X$ as schemes;
    \item If $\Phi:\perf X \to \perf X$ is a birational triangulated equivalence, then for any closed point $x \in X$, there exists a closed point $x' \in X$ such that $\Phi(k(x)) \iso k(x')$. 
\end{enumerate}
In light of Remark \ref{rem: topological reconstruction is enough}, a tt-separated smooth projective variety $X$ can be reconstructed as a connected component of $\spec_{\tens,X}\perf X$ if we can categorically determine $\spc_{\tens,X}\perf X$. 
\end{lemma}
\begin{proof}
    Since the claims here are phrased in a little different ways from \cite{ito2023gluing}, let us comment on how to show this version of the claims although the arguments are essentially same as the proof of \cite{ito2023gluing}*{\href{https://arxiv.org/pdf/2309.08147.pdf}{Lemma 4.20}}. First, note that part (ii) clearly implies part (i). To see part (i) implies part (iii), recall that \cite{ito2023gluing}*{\href{https://arxiv.org/pdf/2309.08147.pdf}{Lemma 4.11 (ii)}} claims for a birational autoequivalence $\Phi:\perf X \simeq \perf X$, $$\Phi\inv(\spc_{\tens_X^\bb L} \perf X) \cap \spc_{\tens_X^\bb L}\perf X$$ agrees with the maximal domain of definition of $\Phi$ (cf. \cite{ito2023gluing}*{\href{https://arxiv.org/pdf/2309.08147.pdf}{Construction 4.6}}). In particular, if $X$ is tt-separated and $\Phi:\perf X \to \perf X$ is a {birational} autoequivalence, then by \cite{ito2023gluing}*{\href{https://arxiv.org/pdf/2309.08147.pdf}{Corollary 4.21}}, the maximal domain of definition of $\Phi$ is the whole $X$ and hence any skyscraper sheaf gets sent to a skyscraper sheaf. Finally, part (iii) clearly implies part (ii) by \cite{ito2023gluing}*{\href{https://arxiv.org/pdf/2309.08147.pdf}{Lemma 4.11}}.
\end{proof}
We can say more about condition (ii) in Lemma \ref{lem: tt-sep}.
\begin{construction}\label{construction: decomposition of tt-seprated variety}
Let $X$ be a tt-separated smooth projective variety. 
Then the condition (iii) in Lemma \ref{lem: tt-sep} and \cite{HuyBook}*{Corollary 5.23} show that {there is the equality of subgroups}
\[
\{\Phi \in \auteq \cal \perf X \mid \Phi \mbox{ is birational up to shift}\} = \pic(X) \ltimes \aut(X) \times \bb Z[1] \subset \auteq \perf X.
\]
Now, consider the set of left cosets
$$
{I_X} := \auteq \perf X/\{\Phi \in \auteq \cal \perf X \mid \Phi \mbox{ is birational up to shift}\}.
$$
By \cite{ito2023gluing}*{\href{https://arxiv.org/pdf/2309.08147.pdf}{Corollary 4.21}, {\href{https://arxiv.org/pdf/2309.08147.pdf}{{Theorem 4.27}}}}, we obtain an isomorphism
$$
\spec_{\tens,X} \perf X \cong \bigsqcup_{{I_X}} X
$$
of schemes.
\end{construction}

\begin{lemma}[\cite{ito2023gluing}*{\href{https://arxiv.org/pdf/2309.08147.pdf}{Lemma 4.30}}]\label{lem: tt-irr}
Let $X$ be a smooth projective variety. The following are equivalent:
\begin{enumerate}
    \item $X$ is tt-irreducible;
    \item For any triangulated equivalence $\Phi:\perf X \simeq \perf X$, we have 
    \[
    \Phi(\spc_{\tens_X^\bb L}\perf X) \cap \spc_{\tens_X^\bb L}\perf X \neq \emp.
    \]
    In particular, any copy of $X$ in $\spec_{\tens, X}\perf X$ intersects with each other. 
    \item Any triangulated equivalence $\Phi:\perf X \simeq \perf X$ is birational up to shift. 
\end{enumerate}
\end{lemma}
\begin{proof}
    Since the wordings are a little different from \cite{ito2023gluing}, let us comment on a proof. First of all, condition (ii) and condition (iii) are equivalent by \cite{ito2023gluing}*{\href{https://arxiv.org/pdf/2309.08147.pdf}{Lemma 4.11 (i)}}. Now, condition (i) and condition (ii) are also equivalent since by \cite{ito2023gluing}*{\href{https://arxiv.org/pdf/2309.08147.pdf}{Theorem 4.27}}, $\spec_{\tens,X} \perf X$ is connected if and only if any Balmer spectra corresponding to tt-structures that are geometric in $X$ intersect with each other, where the former is equivalent to condition (i) by \cite{ito2023gluing}*{\href{https://arxiv.org/pdf/2309.08147.pdf}{Lemma 4.30}} and the latter is equivalent to condition (ii), noting that such Balmer spectra can be mapped to each other by the action of $\auteq \perf X$. 
\end{proof}

Finally, let us list some examples of computations of the Fourier-Mukai locus to advertise what kind of geometry of varieties is reflected in the geometry of the Fourier-Mukai locus. 
\begin{example}[\cite{ito2023gluing}*{\href{https://arxiv.org/pdf/2309.08147.pdf}{Example 1.1, Example 1.4}}]\label{example: FM locus}
    Let $\cal T$ be a triangulated category with a smooth projective variety $X \in \FM \cal T$. 
    \begin{enumerate}
        \item If $X$ is a smooth projective variety with (anti-)ample canonical bundle, then $\spec^\fm\cal T \iso X$. In particular, $X$ is tt-irreducible and tt-separated.
        \item If $X$ is an elliptic curve, then $\spec^\fm \cal T$ is a disjoint union of infinitely many copies of $X$. In particular, $X$ is tt-separated, but not tt-irreducible. 
        \item If $X$ is a simple abelian variety, then all of its copies in $\spec^\fm \cal T$ are disjoint. In particular, $X$ is tt-separated, but not tt-irreducible in general. 
        \item If $X$ is a toric variety, then any copies of $X,Y \in \FM \cal T$ in $\spec^\fm \cal T$ intersect with each other along open sets containing tori. In particular, $X$ is tt-irreducible and not tt-separated in general. 
        \item If $X$ is a surface containing a $(-2)$-curve, then the corresponding spherical twist is birational and $X$ is not tt-separated. Moreover, $X$ is in general not tt-irreducible either. In particular, this shows any del Pezzo surface cannot contain a $(-2)$-curve. 
        \item If $X$ is connected with $X' \in \FM \cal T$ via a standard flop, then at least one pair of their copies in $\spec_\vartriangle \cal T$ intersect with each other along the complement of the flopped subvarieties.
        \item If $X$ is a Calabi-Yau threefold, then each irreducible component of $\spec^\fm \cal T$ containing a copy of $X$ contains all the copies of smooth projective Calabi-Yau threefolds that are birationally equivalent to $X$. Moreover, $X$ is neither tt-separated nor tt-irreducible in general. 
    \end{enumerate}
\end{example}
In the next section, we will generalize parts (ii) and (iii) to all abelian varieties. 
\section{Fourier-Mukai locus of abelian varieties}
In this section, we determine the Fourier-Mukai locus associated to an abelian variety. First, let us recall some basics of the derived category of coherent sheaves on an abelian variety. {For the rest of this paper, $k$ is an algebraically closed field of characteristic $0$.} 
\begin{notation}
    Let $X$ be an abelian variety and let $\hat X$ denote its dual. For a closed point $x \in X$, let $t_x:X\overset{\sim}{\to} X; y \mapsto y + x$ denote the translation. Moreover, for a closed point $\alpha \in \hat X$, let $\ecal L_\alpha \in \pic^0(X)$ denote the corresponding line bundle of degree $0$.   
\end{notation}
In \cite{Orlov_2002}, Orlov gave several important results on the derived category of coherent sheaves on an abelian variety. 
\begin{definition}\label{def:symplectic}
    Let $X$ be an abelian variety. Then, define the group of \textbf{symplectic automorphisms} of $X \times \hat X$ (with respect to the natural symplectic form) to be 
    \[
    \sp (X\times \hat X):= \l\{\begin{pmatrix}
        f_1 & f_2 \\ f_3 & f_4
    \end{pmatrix} \in \aut(X \times \hat X) \middle| \begin{pmatrix}
        f_1 & f_2 \\ f_3 & f_4
    \end{pmatrix}\begin{pmatrix}
        \hat f_4 & - \hat f_2 \\ -\hat f_3 & \hat f_1
    \end{pmatrix} = \id_{X\times \hat X} \r\},
    \]
    where $\aut(X \times \hat X)$ denotes the group of automorphisms of abelian varieties and $\hat f_i$ denotes the transpose of $f_i$. Here, we are writing an automorphism $f:X\times \hat X \to X\times \hat X$ with matrix form, where $f_1:X \to X$, $f_2:\hat X \to X$, etc. We say a symplectic automorphism $f$ is \textbf{elementary} if $f_2$ is an isogeny.  
\end{definition}
\begin{theorem}[\cite{Orlov_2002}*{\href{https://arxiv.org/pdf/alg-geom/9712017.pdf}{Theorem 2.10}, \href{https://arxiv.org/pdf/alg-geom/9712017.pdf}{Corollary 2.13}, \href{https://arxiv.org/pdf/alg-geom/9712017.pdf}{Proposition 3.2}, \href{https://arxiv.org/pdf/alg-geom/9712017.pdf}{Construction 4.10}, \href{https://arxiv.org/pdf/alg-geom/9712017.pdf}{Proposition 4.12}}]\label{thm: Orlov abelian}
    Let $X$ be an abelian variety. Then, there is a 
    group homomorphism
    \[
    \gamma:\auteq \perf X \to \sp(X\times \hat X)
    \]
    such that for any $\Phi_\ecal E \in \auteq \perf X$ with Fourier-Mukai kernel $\ecal E \in \perf (X\times X)$ (which is necessarily isomorphic to a sheaf on $X\times X$ up to shift) and for any $(a,\alpha),(b,\beta) \in X\times \hat X$, we have that $\gamma(\Phi_\ecal E)(a,\alpha) = (b,\beta)$ if and only if 
    \[
    {t_{(0,b)}}_*\ecal E \tens_{\ecal O_{X\times X}} \pi_2^*{\ecal L}_\beta \iso t^*_{(a,0)} \ecal E\tens_{\ecal O_{X\times X}} \pi_1^* \ecal L_\alpha
    \]
    where $\pi_i$ denotes projections $\pi_i:X\times X \to X$ so that $\Phi_\ecal E(-) = {\bb R\pi_2}_*(\pi_1^*(-)\tens_{\ecal O_{X\times X}}^\bb L \ecal E)$. Moreover, we have
    \[
    \ker \gamma = (X\times \hat X)_k\times \bb Z[1] \subset \auteq \perf X,
    \]
    where each component corresponds to translations, tensor products with lines bundles of degree $0$, and shifts, respectively. Furthermore, for any elementary symplectic automorphism $f \in \sp(X\times \hat X)$, there exists a (semihomogeneous) vector bundle $\ecal E$ on $X \times X$ such that $\gamma(\Phi_\ecal E) = f$. 
\end{theorem}
We have the following "global" understanding of the Fourier-Mukai locus of an abelian variety:
\begin{lemma}[\cite{ito2023gluing}*{\href{https://arxiv.org/pdf/2309.08147.pdf}{Lemma 5.1}}]\label{lem: Ito}
Let $\cal T$ be a triangulated category with an abelian variety $X \in \FM \cal T$. Then, any $Y \in \FM \cal T$ is also an abelian variety and we have
\[
\spec^\fm \cal T =\bigsqcup _{Y \in \FM \cal T} \spec_{\tens,Y} \cal T.
\]
as a scheme. 
\end{lemma}
In particular, in order to understand the Fourier-Mukai locus, we can focus on the locus $\spec_{\tens,X} \cal T$ for a single abelian variety $X \in \FM \cal T$. In \cite{ito2023gluing}*{\href{https://arxiv.org/pdf/2309.08147.pdf}{Lemma 5.2}}, the following claim was only shown for an abelian variety with isomorphic dual, but it is also straightforward to show the result in general:
\begin{prop}
    An abelian variety $X$ is not tt-irreducible. 
\end{prop}
\begin{proof}
    By Lemma \ref{lem: tt-irr}, it suffices to show there is a triangulated equivalence $\Phi:\perf X \to \perf X$ such that
    \[
    \Phi(\spc_{\tens_X^\bb L}\perf X) \cap \spc_{\tens_X^\bb L}\perf X = \emp.
    \]
    Take an ample line bundle $\ecal L$ on $\hat X$ with isogeny $\phi_\ecal L: \hat X \to  \hat{\hat X} \iso X, x \mapsto t_x^*\ecal L \tens \ecal L^{-1}$. First, note that by \cite{Ploog_2005}*{\href{https://www.mathematik.hu-berlin.de/~ploog/PAPERS/Ploog_phd.pdf}{Example 4.5}}, we have 
    \[
     \gamma(-\tens_{\ecal O_{\hat X}}\ecal L) = \begin{pmatrix}
        \id_{\hat X}  &  0 \\
        \phi_\ecal L & \id_{X}
    \end{pmatrix} \in \sp(\hat X \times X)
    \]
    and in particular $\phi_\ecal L = \hat{\phi_{\ecal L}}$. Thus, we have 
    \[
    f:=\begin{pmatrix}
        \id_X  &  \phi_\ecal L \\
        0 & \id_{\hat X}
    \end{pmatrix} \in \sp(X\times \hat X).
    \]
    Now, since $f$ is, in particular, an elementary symplectic isomorphism, we have a vector bundle $\ecal E$ on $X\times X$ such that $\Phi_\ecal E \in \auteq \perf X$ (with $\gamma(\Phi_\ecal E) = f$) by Theorem \ref{thm: Orlov abelian}. Therefore, we see that 
    \[
    \Phi_{\ecal E}(\spc_{\tens_X^\bb L}\perf X) \cap \spc_{\tens_X^\bb L}\perf X = \emp.
    \] 
    by \cite{ito2023gluing}*{\href{https://arxiv.org/abs/2309.08147}{Corollary 4.10}} as desired. 
\end{proof}
Now, the following result gives an affirmative answer to \cite{ito2023gluing}*{\href{https://arxiv.org/abs/2309.08147}{Conjecture 5.9}}:
\begin{theorem}\label{main theorem: abelian}
    An abelian variety $X$ is tt-separated. 
\end{theorem}
\begin{proof}
    Take a birational autoequivalence $\Phi\in \auteq \perf X$, i.e., suppose there exist $x_0,y_0\in X$ such that $\Phi(k(x_0)) \iso k(y_0)$. By Lemma \ref{lem: tt-sep}, it suffices to show that for any $x \in X$, there exists $y \in X$ such that $\Phi(k(x)) \iso k(y)$. Now, by \cite{martin2013relative}*{\href{https://arxiv.org/pdf/1011.1890.pdf}{Proposition 3.2}}, we have $\Phi = \Phi_\ecal K$ for a sheaf $\ecal K$ on $X \times X$ that is flat along each projection and therefore  by \cite{HuyBook}*{Example 5.4} we have $\ecal K|_{\{x_0\}\times X} \iso \Phi(k(x_0) \iso k(y_0)$ under the canonical identification $\{x_0\}\times X \iso X$. Moreover, by  Theorem \ref{thm: Orlov abelian}, there is a corresponding isomorphism $f_\ecal K:X\times \hat X \to X\times \hat X$ satisfying $f_\ecal K(a,\alpha) = (b,\beta)$ if and only if 
    \[
    t^*_{(a,0)} \ecal K \tens_{\ecal O_{X\times X}} \pi_1^* \ecal L_\alpha \iso { t_{(0,b)}}_*\ecal K  \tens_{\ecal O_{X\times X}} \pi_2^* \ecal L_\beta,
    \]
    which is equivalent to
    \[
    t^*_{(a,0)}\ecal K \iso t^*_{(0,-b)}\ecal K \tens_{\ecal O_{X\times X}} \pi_2^*\ecal L_\beta \tens_{\ecal O_{X\times X}} (\pi_1^* \ecal L_\alpha) \inv. 
    \]
    Under canonical identifications $\{x\} \times X \iso X \iso X \times \{y\}$ for closed points $x,y \in X$, we therefore see that for any $a \in X$, we can take $\alpha, b,\beta$ such that 
    \begin{align*}
        \Phi(k(x_0+a)) &\iso \ecal K|_{\{x_0+a\}\times X} \iso (t^*_{(a,0)}\ecal K)|_{\{x_0\}\times X} \\
        &\iso \l(t^*_{(0,-b)}\ecal K \tens_{\ecal O_{X\times X}} \pi_2^*\ecal L_\beta \tens_{\ecal O_{X\times X}} (\pi_1^* \ecal L_\alpha)\inv \r)|_{\{x_0\}\times X}  \\
        & \iso t^*_{-b}(\ecal K|_{\{x_0\}\times X}) \tens_X \l( \pi_2^*\ecal L_\beta \tens_{\ecal O_{X\times X}} (\pi_1^* \ecal L_\alpha)\inv \r)|_{\{x_0\}\times X} \\
        & \iso k(y_0+b) \tens _{\ecal O_X}\l( \pi_2^*\ecal L_\beta \tens_{\ecal O_{X\times X}} (\pi_1^* \ecal L_\alpha)\inv \r)|_{\{x_0\}\times X}\iso k(y_0 + b)
    \end{align*}
    as desired. 
\end{proof}

{As a combination of Construction \ref{construction: decomposition of tt-seprated variety}, Lemma \ref{lem: Ito}, and Theorem \ref{main theorem: abelian}, we obtain some kind of a generalization of \cite{matsui2023triangular}*{Corollary 4.10} (see also \cite{HO22}*{Theorem 4.11} and \cite{hirano2024FMlocusK3}*{Theorem 5.3}). Here we note that $\spec^\fm \perf X = \spec_\vartriangle \perf X$ holds {for an elliptic curve $X$} by \cite{HO22}*{Proposition 4.10} and the subsequent argument. 
\begin{corollary}\label{cor: abelian}
Let $X$ be an abelian variety.
Then there is an isomorphism
$$
\spec^\fm \perf X \cong \bigsqcup _{Y \in \FM {\perf X}} \bigsqcup_{{I_Y}} Y
$$
of schemes.
\end{corollary}
}

By Theorem \ref{prop: structure sheaf of FM locus} {and Corollary \ref{cor: abelian}}, we can reconstruct all the Fourier-Mukai partners of an abelian variety $X$ if we can identify the Fourier-Mukai locus 
\[
\spc^\fm \perf X \subset \spc_\vartriangle \perf X
\]
purely categorically. Along this line, the following conjecture was made in \cite{ito2023gluing}:
\begin{conjecture}[\cite{ito2023gluing}*{\href{https://arxiv.org/pdf/2309.08147}{Conjecture 6.14}}]
    Let $\cal T$ be a triangulated category with $\FM \cal T\neq \emp$. Then, we have
    \[
    \spc^\fm \cal T = \spc^\ser \cal T. 
    \]
\end{conjecture}
The conjecture holds for curves (in particular, elliptic curves) and smooth projective varieties with (anti-) ample canonical bundle, but in \cite{hirano2024FMlocusK3}, it was shown that when we have a certain K3 surface $X \in \FM \cal T$, the conjecture fails (\cite{hirano2024FMlocusK3}*{\href{https://arxiv.org/pdf/2405.01169}{Theorem 5.8}}). Their proof relies on the existence of spherical objects in $\cal T$ so the result does not directly generalize to abelian surfaces and we are interested if there are certain classes of abelian varieties of dimension $>1$ for which the conjecture holds.  
\bibliography{bib}
\end{document}